\documentclass[11pt, leqno]{amsart}

\usepackage{amsmath,amssymb,amsthm, epsfig}

\usepackage{amsmath}

\usepackage{amssymb}

\usepackage{color}

\usepackage{ulem}

\usepackage{epsfig}

\usepackage[mathscr]{eucal}

\usepackage{mathptmx} 
\usepackage[scaled=1]{helvet} 
\usepackage{courier} 
\normalfont
\usepackage[T1]{fontenc}

\title[Regularity for flat solutions]{Asymptotics and regularity of flat solutions to fully nonlinear elliptic problems}
\author{Disson dos Prazeres}
\author{Eduardo V. Teixeira }

\address{Universidade Federal do Cear\'a, Departamento de Matem\'atica, Campus do Pici - Bloco 914 , Fortaleza, CE - Brazil 60.455-760.}

\address{Universidade Federal do Cear\'a, Departamento de Matem\'atica, Campus do Pici - Bloco 914 , Fortaleza, CE - Brazil 60.455-760.}

\email{teixeira@mat.ufc.br} 

\date{}

\newtheorem{theorem}{Theorem}[section]

\newtheorem{lemma}[theorem]{Lemma}

\newtheorem{corollary}[theorem]{Corollary}

\theoremstyle{definition}

\newtheorem{definition}[theorem]{Definition}

\theoremstyle{remark}

\newtheorem{remark}[theorem]{Remark}

\numberwithin{equation}{section}

\newcommand{\intav}[1]{\mathchoice {\mathop{\vrule width 6pt height 3 pt depth  -2.5pt
\kern -8pt \intop}\nolimits_{\kern -6pt#1}} {\mathop{\vrule width
5pt height 3  pt depth -2.6pt \kern -6pt \intop}\nolimits_{#1}}
{\mathop{\vrule width 5pt height 3 pt depth -2.6pt \kern -6pt
\intop}\nolimits_{#1}} {\mathop{\vrule width 5pt height 3 pt depth
-2.6pt \kern -6pt \intop}\nolimits_{#1}}}

\begin{document}

\begin{abstract}

In this work we establish local $C^{2,\alpha}$ regularity estimates for flat solutions to non-convex fully nonlinear elliptic equations provided the coefficients  and the source function are of class $C^{0,\alpha}$. For problems with merely continuous data, we prove that flat solutions are locally $C^{1,\text{Log-Lip}}$.
\medskip


\medskip


\end{abstract}

\maketitle

\tableofcontents

\section{Introduction}

The goal of this paper is to obtain optimal estimates for flat solutions to a class of non-convex fully nonlinear elliptic equations of the form
\begin{equation}
 F(X,D^2u) =  \mathscr{G}(X, u, \nabla u). \label{1.1prima}
\end{equation}
Under continuous differentiability with respect to the matrix variable and appropriate continuity assumptions on the coefficients and on the source function, we present a Schauder type regularity result for flat solutions, namely for solutions with small enough norm, $|u|\ll 1$.

\par
\medskip

The nonlinear operator $F \colon B_{1}\times \text{Sym}(n)\rightarrow \mathbb{R}$ is assumed to be uniformly elliptic, namely, there exist constants $ 0<\lambda \leq \Lambda$ such that for any $M, P \in \text{Sym}(n)$, with $P\ge 0$ and all $X \in B_{1} \subset \mathbb{R}^n$ there holds
\begin{equation} \label{ellip}
\lambda\|P\|\leq F(X,M+P)-F(X,M)\leq\ \Lambda\|P\|.
\end{equation}
Under such condition it follows as a consequence of Krylov-Safonov Harnack inequality that solutions to the homogeneous, constant coefficient equation
\begin{equation} \label{ellip-hom}
    F(D^2h) = 0
\end{equation}
are locally of class $C^{1,\alpha}$, for some $0 < \alpha < 1$. Under appropriate hypotheses on $\mathscr{G} \colon B_1 \times \mathbb{R} \times \mathbb{R}^n \to \mathbb{R}$, the same conclusion is obtained, i.e., viscosity solutions are of class $C^{1,\alpha}$. Thus, insofar as the regularity theory for equation of the form \eqref{1.1prima} is concerned, one can regard the right hand side $\mathscr{G}(X, u, \nabla u)$ as an $\tilde{\alpha}-$H\"older continuous source, $f(X)$. Therefore, within this present work, we choose to look at the RHS $\mathscr{G}(X, u, \nabla u)$ simply as a source term $f(X)$, and equation \eqref{1.1prima} will be written as
\begin{equation}
 F(X,D^2u) = f(X). \label{1.1}
\end{equation}

\par
\medskip

Regularity theory for heterogeneous equations \eqref{1.1} has been a central target of research for the past three decades.  While a celebrated result due to Evans and Krylov assures that solutions to convex equations are classical, i.e., $C^{2,\alpha}$ for some $\alpha>0$, the problem of establishing continuity of the Hessian of solutions to general equations of the form \eqref{ellip-hom} challenged the community for over twenty years. The problem has been settled in the negative by Nadirashvili and Vladut, \cite{NV1, NV2}, who exhibit solutions to uniform elliptic equations whose Hessian blows-up.  

\par
\medskip

In view of the impossibility of a general existence theory for classical solutions to all fully nonlinear equations \eqref{ellip-hom}, it becomes a central topic of research the study of reasonable conditions on $F$ and on $u$ as to assure the Hessian of the solution is continuous. In such perspective the works \cite{CY} and  \cite{CC03}  on interior $C^{2,\alpha}$ estimates for a particular class of non-convex equations are highlights. Another work towards Hessian estimates of solutions to fully nonlinear elliptic equations is \cite{S}, where it is proven that if the operator is of class $C^2$ in all of its arguments, then small solutions are classical.

\par
\medskip

Inspired by problems of the form \eqref{1.1prima}, in the present work, we obtain regularity estimates for flat solutions to heterogeneous equation \eqref{1.1}, under continuity conditions on the media. We show that if $X \mapsto (F(X, \cdot), f(X))$ is $\alpha$-H\"older continuous, then flat solutions are locally $C^{2, \alpha}$. In the case $\alpha = 0$, namely when the coefficients and the source are known to be just continuous, we show that flat solutions are locally $C^{1,\text{Log-Lip}}$. 

\par
\medskip

The proofs of both results mentioned above, to be properly stated in Theorem \ref{teo1}
and Theorem \ref{thm LL} respectively,  are based on a combination of 
geometric tangential analysis and perturbation arguments inspired by compactness methods
in the theory of elliptic PDEs. 

\par
\medskip

We conclude this introduction explaining the heuristics of the geometric tangential analysis behind our proofs.  Given a fully nonlinear elliptic operator $F$, we look at the family of elliptic scalings 
$$
	F_\mu(M) := \dfrac{1}{\mu} F(\mu M), \quad \mu >0.
$$
This is a continuous family of operators preserving the ellipticity constants of the original equation. If $F$ is differentiable at the origin (recall, by normalization $F(0) = 0$), then indeed 
$$
	F_\mu (M) \to \partial_{M_{ij}}F(0)M_{ij}, \quad \text{as } \mu \to 0.
$$
In other words, the linear operator  $M \mapsto \partial_{M_{ij}}F(0)M_{ij}$ is the \textit{tangential equation} of $F_\mu$ as $\mu \to 0$. Now, if $u$ solves an equation involving the original operator $F$, then $u_\mu := \frac{1}{\mu} u$ is a solution to a related equation for $F_\mu$. However, if in addition it is known that the norm of $u$ is at most $\mu$, then it accounts into saying that $u_\mu$ is a normalized solution to the $\mu$-related equation, and hence we can access the universal regularity theory available for the (linear) tangential equation by compactness methods. In the sequel we transport such  good {\it limiting} estimates towards $u_\mu$, properly adjusted by the geometric tangential path used to access the {\it tangential} linear elliptic regularity theory. 

\par
\medskip

The paper is organized as follows. In Section \ref{sec hyp and thm} we state all the hypotheses, mathematical set-up and notions to be used throughout the whole paper. In that Section we also state properly the two main Theorems proven in the work.  In Section \ref{sec GTO} we rigorously develop the heuristics of the geometric tangential analysis explained in the previous paragraph.  The proof of $C^{2,\alpha}$ estimates, Theorem \ref{teo1}, will be delivered in Section
\ref{Sect. c2alpha}. Two applications of such a result will  be discussed in Section \ref{sec appl}. Theorem \ref{thm LL} will be proven in
Section \ref{Sect. Log-Lip}.

\section{Hypotheses and main results} \label{sec hyp and thm}

Let us start off by discussing the  hypotheses, set-up and main notations used
in this article. For $B_1$ we denote the open unit ball in the
Euclidean space $\mathbb{R}^n$. The space of  $n\times n$
symmetric matrices will be denoted by $\text{Sym}(n)$. By modulus
of continuity we mean an increasing function $\varpi \colon [0, +\infty) \to [0,+\infty) $,  with $\varpi(0)=0$. 

\par
\medskip

Hereafter we shall assume the following conditions on the operator $F \colon
B_1 \times \text{Sym}(n) \to \mathbb{R}$ and $f\colon B_1 \to
\mathbb{R}$:
\begin{enumerate}
    \item[(H1)] There exist constants $ 0<\lambda \leq \Lambda$ such that for any $M, P \in\text{Sym}(n) $, with $P\ge 0$ and all $X \in B_{1}$,  there holds
\begin{equation} \label{ellip-local}
\lambda\|P\|\leq F(X,M+P)-F(X,M)\leq\ \Lambda\|P\|.
\end{equation}
\item[(H2)] $F(X, M)$ is differentiable with respect to $M$ and for a modulus of continuity $\omega$, there holds
\begin{equation}\label{C1}
    \|D_{M}F(X,M_1)-D_{M}F(X,M_2)\| \leq\omega(\|M_1-M_2\|),
\end{equation}
for all $(X,M_i) \in B_1 \times \text{Sym}(n)$.
    \item[(H3)]  For another modulus of continuity $\tau$, there holds
\begin{eqnarray}
    |F(X,M)-F(Y,M)| &\leq & \tau(|X-Y|) \cdot \|M\|,    \label{cont coeff} \\
    |f(X) - f(Y)|         &\leq  & \tau(|X-Y|),                     \label{cont source}
\end{eqnarray}
\end{enumerate}
for all $X,Y\in B_{1}$  and  $M \in \text{Sym}(n)$. It will also be enforced hereafter in this paper the following normalization conditions:
\begin{equation} \label{normalization}
    F(0, 0_{n\times n}) =  f(0) = 0;
\end{equation}
though such hypothesis is not restrictive, as one can always reduce the problem as to verify that. 

\par
\medskip

Condition (H1) concerns the notion of uniform ellipticity. Under such a structural condition, the theory of viscosity solutions provides an appropriate notion for weak solutions to such equations.

\begin{definition} A continuous function $u \in C^0(B_1)$ is said to be a viscosity subsolution to \eqref{1.1} in $B_1$ if whenever one touches the graph of $u$ by above by a smooth function $\varphi$ at $X_0 \in B_1$, there holds
$$
    F(X_0, D^2\varphi(X_0)) \ge f(X_0).
$$
Similarly, $u$ is a viscosity supersolution to \eqref{1.1} if whenever one touches the graph of $u$ by below by a smooth function $\phi$ at $Y_0 \in B_1$, there holds
$$
    F(Y_0, D^2\phi(Y_0)) \le f(Y_0).
$$
We say $u$ is a viscosity solution to \eqref{1.1} if it is a subsolution and a supersolution of \eqref{1.1}.

\end{definition}

\par
\medskip

Condition (H2) fixes a modulus of
continuity $\omega$ to the derivative of $F$. The regularity
estimates proven in this paper depends upon $\omega$. Condition
(H3) sets the continuity of the media. When $\tau(t) \approx
t^\alpha$, $0 < \alpha < 1$, the coefficients and the source
function are said to be $\alpha$-H\"older continuous.  In such
scenario we prove that flat solutions are locally of class
$C^{2,\alpha}$ -- a sharp Schauder type of estimate for non-convex fully nonlinear equations.

\begin{theorem}[$C^{2,\alpha}$ regularity] \label{teo1} Let $u\in C^{0}\left( B_{1}\right) $ be a viscosity solution
to
$$
F(X,D^{2}u) = f(X)\text{ in }B_{1},
$$
where $F$ and $f$ satisfy (H1)--(H3) with $\tau(t) = C t^\alpha$
for some $0< \alpha < 1$. There
exist a $ \overline{\delta} > 0$, depending only upon $n,\lambda
,\Lambda ,\omega, \alpha$, and $\tau(1)$, such that if
$$
    \sup_{B_{1}}| u | \leq \overline{\delta}
$$
then $u\in C^{2, \alpha } (B_{ {1}/{2}} )$ and
$$
    \|u\|_{C^{2,\alpha}(B_{1/2})}  \le M \cdot \overline{\delta},
$$
where $M$ depends only upon $n,\lambda ,\Lambda ,\omega$, and  $(1-\alpha) $.
\end{theorem}


\par
\medskip

If $f$ is merely continuous, then even for the classical Poisson equation
$$
    \Delta u = f(X),
$$
solutions may fail to be of class $C^2$. In connection to Theorem 5.1 in \cite{T1}, in this paper we show that flat solutions in continuous media are locally of class $C^{1, \text{Log-Lip}}$, which corresponds to the optimal regularity estimate under such weaker conditions.

\begin{theorem}[$C^{1, \text{Log-Lip}}$ estimates]
\label{thm LL}Let $u\in C^{0}(B_{1})$ be a viscosity solution to
$$
F(X,D^{2}u)=f(X)\mbox{ in }B_{1}.
$$
Assume (H1)--(H3). Then there exist a $\overline{\delta
}=\overline{\delta }(n,\lambda,\Lambda ,\omega ,\tau)$ such that
if
$$
\sup_{B_{1}}\|u\|\leq\overline{\delta},
$$
then $u\in C^{1, \mathrm{Log-Lip}}(B_{\frac{1}{2}})$ and
$$
    \left | u(X) - \left [ u(Y) + \nabla u(Y) \cdot (X-Y) \right ] \right | \le -M \overline{\delta } \cdot |X-Y|^2 \log (|X-Y|),
$$
for a constant $M$ that depends only upon  $n,\lambda ,\Lambda ,\omega$, and  $(1-\alpha)$.
\end{theorem}

\section{Geometric tangential analysis} \label{sec GTO}

In this Section we provide a rigorous treatment of the heuristics involved in the geometric tangential analysis explained at the end of the Introduction. The next Lemmas are central for the proof of both Theorem \ref{teo1} and Theorem \ref{thm LL}.

\begin{lemma}\label{Key lemma0} Let $F \colon B_1 \times  \text{Sym}(n)\rightarrow \mathbb{R}$ satisfy conditions (H1) and (H2). Given $0  \le \gamma < 1$, there exists $\eta > 0$, depending only on $n,\lambda ,\Lambda, \omega$, and  $\gamma$,  such that if $u$ satisfies $|u|\le 1$ in $B_1$ and solves $\mu^{-1} F(X, \mu D^2u) = f(X)$ in $B_1$, for
$$
   0 < \mu \le \eta, \quad  \sup_{M \in \text{Sym}(n)} \frac{|F(X,M)-F(0,M)|}{\|M\|} \leq \eta \quad  \mbox{ and } \quad  \| f\|_{L^{\infty }(B_{1})}\leq \eta,
$$
then one can find a number $0< \sigma < 1$, depending only on $n, \lambda$ and $\Lambda$, and  a quadratic polynomial $P$ satisfying
$$
\mu^{-1} F(0,\mu D^{2}P)=0, \quad \text{ with } \quad \|P\|_{L^{\infty }(B_{1})} \le C(n, \lambda, \Lambda), 
$$
for a universal constant  $C(n, \lambda, \Lambda) > 0$, such that
$$
\sup_{B_{\sigma }}|u-P| \leq\sigma ^{2+\gamma}.
$$
\end{lemma}

\begin{proof} Let us suppose, for the purpose of contradiction, that the Lemma fails to hold. If so, there would exist a sequence of elliptic operators, $F_k(X,M)$,
satisfying  hypotheses (H1) and (H2), a sequence $0 < \mu_k = \text{o}(1)$  and sequences of functions 
$$
	u_{k}\in C(B_{1}) \mbox{ and }f_{k}\in L^{\infty }(B_{1}),
$$ 
all  linked through the equation
\begin{equation} \label{eq0 KL0}
\dfrac{1}{\mu_k} F_{k}(X, \mu_k D^{2}u_{k})=f_{k}(X)\mbox{ in }B_{1},
\end{equation}
in the viscosity sense, such that
\begin{equation}
  \|u_{k}\| _{\infty }\leq 1, \quad \mu_k \le \frac{1}{k}, \quad  \sup_{M \in \text{Sym}(n)} \frac{ |F_{k}(X,M)-F_{k}(0,M)|}{\|M\|} \leq \dfrac{1}{k} \quad \mbox{ and } \quad  \|f_{k}\| _{\infty }\leq \frac{1}{k}; \label{eq1 KL0}
\end{equation}
however for some $0< \sigma_{0} <1$
\begin{equation}
    \sup_{B_{\sigma_{0}}}|u_{k}-P|>\sigma_{0}^{2+\gamma} ,\label{eq01 KL00}
\end{equation}
that for all quadratic polynomials $P$ that satisfies
$$
    \dfrac{1}{\mu_k} F_k(0, \mu_k D^2P) = 0.
$$
Passing to a subsequence if necessary, we can assume $F_k (X, M) \to F_\infty (X, M)$ locally uniform in $\text{Sym}(n)$. From uniform $C^1$ estimate on $F_k$ and the coefficient oscillation hypothesis in \eqref{eq1 KL0}, we deduce
\begin{equation} \label{eq2 KL2}
\dfrac{1}{\mu_k} F_{k}(X, \mu_k M) \to D_MF_\infty(0,0)\cdot M,
\end{equation}
locally uniform in $\text{Sym}(n)$. Also, by $C^{1, \alpha}$ {\it a priori} estimates for equation \eqref{eq0 KL0}, up to a subsequence, $u_k \to u_\infty$ locally uniform in $B_1$. Thus, by stability of viscosity solutions, we conclude
\begin{equation} \label{eq0 KL3}
	D_MF_\infty(0,0)\cdot D^2u_\infty = 0, \quad \text{ in } B_1.
\end{equation}
As $u_\infty$ solves a linear, constant coefficient elliptic equation, $u_\infty$ is smooth. Define
$$
    P:=u_{\infty}(0)+ Du_{\infty}(0) \cdot X + \dfrac{1}{2} X .D^{2}u_{\infty}(0)X.
$$
Since $\|u_\infty\|\le 1$, it follows from $C^3$ estimates on $u_\infty$ that
$$
\sup_{B_{r}}|u_{\infty}-P|\leq  Cr^3   \label{1.2prima},
$$
for a constant $C$ that depends only upon dimension $n$ and ellipticity constants, $\lambda$ and $\Lambda$.  Thus, if we select
$$
    \sigma := \sqrt[1-\gamma]{\dfrac{1}{2C}},
$$
a choice that depends only on $n$,  $\lambda$, $\Lambda$ and $\gamma$, we readily have
$$
\sup_{B_{\sigma}}|u_{\infty}-P|\leq  \dfrac{1}{2} \sigma^{2+\gamma}   \label{eq4 KL0},
$$
Also, from equation \eqref{eq0 KL3}, we obtain
$$
    D_{M}F_{\infty}( 0,0) \cdot D^{2 }P = 0
$$
which implies that
$$
\vert \mu_k^{-1} F_{k}(0, \mu_k D^2P 	 \vert= \text{o}(1).
$$
Now, since $F_{k}$ is uniformly elliptic in $B_1 \times
\text{Sym}(n)$ and $F_{k}(0,0)=0$, it is possible to find a
sequence of real numbers $(a_{k})\subset\mathbb{R}$ with $|a_{k}|=
\text{o}(1)$, for which the quadratic polynomial
$$
    P_{k}: = P + a_{k}|X|^{2}
$$
do satisfy
$$
    \mu_k^{-1} F_{k}(0, \mu_k D^{2}P_{k})=0.
$$
Finally we have, for any point in $B_{\sigma}$ and $k$ large enough,
$$
\begin{aligned}
 \sup\limits_{B_{\sigma}} |u_{k}-P_{k}|& \leq |u_{k} - u_{\infty}| + | u_\infty - P| + |P - P_{k}| \\
&\leq  \dfrac{1}{5} \sigma^{2+\gamma} + \dfrac{1}{2} \sigma^{2+\gamma} 
+ |a_{k}| \sigma^{2} \\
& <  \sigma^{2+ \gamma},
\end{aligned}
$$
which contradicts \eqref{eq01 KL00}. Lemma \ref{Key lemma0} is proven.
\end{proof}

In the sequel, we transfer the geometric tangential access towards a smallness condition of the $L^\infty$ of the solution.

\begin{lemma}\label{Key lemma} Let $F$ satisfy (H1) and (H2) and $0\le \alpha < 1$ be  given. There exist small a positive constant $\delta>0$ depending on $n,\lambda ,\Lambda $,and $\alpha$, and a constant  $0< \sigma < 1$ depending only on $n,\lambda ,\Lambda$ and $(1-\alpha)$ such that if $u$ is a solution to (\ref{1.1}) and
$$
    \| u \|_{L^{\infty }(B_{1})}\leq \delta, \quad \sup_{M \in \text{Sym}(n)} \frac{|F(X,M)-F(0,M)|}{\|M\|} \leq \delta^{3/2}  \quad  \mbox{ and } \quad  \| f\|_{L^{\infty }(B_{1})}\leq \delta ^{3/2 },
$$
then one can find a quadratic polynomial $P$ satisfying
\begin{equation}
F(0,D^{2}P)=0, \quad \text{ with } \quad \|P\|_{L^{\infty }(B_{1})} \le \delta C(n, \lambda, \Lambda) \label{P-soluçaodoproblemahomogeneo}
\end{equation}
for a universal constant  $C(n, \lambda, \Lambda) > 0$, and
$$
\sup_{B_{\sigma }}|u-P| \leq \delta \cdot \sigma ^{2+\alpha}  
$$
\end{lemma}
\begin{proof} Define the normalized function $v = \delta^{-1}  u$. We immediate check that
$$
	 \delta^{-1} F(X,  \delta D^2v) = \dfrac{f(X)}{\delta}.
$$
If $\eta$ is the number from Lemma \ref{Key lemma0}, we choose $\delta =  \eta^2$ and the Lemma follows.
\end{proof}

\section{$C^{2,\alpha}$ estimates in $C^{0,\alpha}$ media} \label{Sect. c2alpha}

In this Section we show that if the coefficients and the source are $\alpha$-H\"older continuous, then flat solutions are locally of class $C^{2,\alpha}$, i.e.  That is, herein we assume
\begin{equation}\label{hyp Calpha media}
     \tau(t) \lesssim C t^\alpha,
\end{equation}
for some $0< \alpha < 1$ and $C>0$, where $\tau$ is the modulus of continuity of the coefficients and the source function appearing in \eqref{cont coeff} and \eqref{cont source}.  Under such condition, we aim to show that flat solutions are locally of class $C^{2,\alpha}$.  
\medskip

The idea of the proof is to employ Lemma \ref{Key lemma} in an inductive process as to establish the aimed $C^{2,\alpha}$ estimate for flat solutions under an appropriate  smallness regime for the oscillation of the coefficients and the source function.
\begin{lemma}
\label{teo1auxiliar} Let $F$, $f$ and $u$ be under the hypotheses
of Lemma \ref{Key lemma}. Then there
exists a $\delta =\delta(n,\lambda ,\Lambda ,\omega)>0$,
such that if
$$
\sup_{B_{1}}\vert u\vert \leq \delta \quad \text{ and } \quad  \tau(1) \le \delta^{3/2},
$$
then $u\in C^{2, \alpha}$ at the origin and
$$
    |u - (u(0) + \nabla u(0)\cdot X + \frac{1}{2} X^t D^2u(0) X) | \le C \cdot \delta |X|^{2+\alpha},
$$
where $C>0$ depends only upon $n,\lambda ,\Lambda ,\omega$ and $(1-\alpha)$.
\end{lemma}

\begin{proof}
The proof consists in iterating Lemma \ref{Key lemma} as to produce a sequence of quadratic polynomials
\begin{equation} \label{1.3prima}
    P_{k}=\frac{1}{2}X^{t}A_{k}X+b_{k}\cdot X +c_{k} \quad \text{ with } \quad F(0, D^2P_k) = 0,
\end{equation}
that approximates $u$ in a $C^{2,\alpha}$ fashion, i.e.,
\begin{equation}
\sup\limits_{B_{\sigma^k}} \vert u(X)-P_{k}(X)\vert \leq \delta\sigma ^{(2+\alpha)k}.  \label{1.3}
\end{equation}
Furthermore, we aim to control the oscillation of the coefficients of $P_k$ as
\begin{equation}
\left\{
\begin{array}{rrl}
|A_{k}-A_{k-1}| & \leq & C\delta \sigma^{\alpha(k-1)}\\
|b_{k}-b_{k-1}| & \leq & C\delta \sigma^{(1+\alpha)(k-1)} \\
|c_{k}-c_{k-1}| & \leq & C\delta \sigma^{(2+\alpha)(k-1)}
\end{array}\right.\label{2.13}
\end{equation}
where $C>0$ is universal and  $\sigma$ and $\delta$ are the parameters from Lemma \ref{Key lemma}. The proof of existence of polynomials $P_k$ verifying \eqref{1.3prima},  \eqref{1.3} and \eqref{2.13} will be delivered by induction. The case $k=1$ is precisely the statement of Lemma \ref{Key lemma}.  Suppose now we have verified the $k$th step of induction, i.e., by  there exists a quadratic polynomial $P_{k}$ satisfying \eqref{1.3prima},  \eqref{1.3} and \eqref{2.13}. We define
\begin{eqnarray}
    \tilde{u}(X)&:= & \dfrac{1}{\sigma ^{(2+\alpha) k}}(u(\sigma ^{k}X)-P_{k}(\sigma^{k}X)); \label{def utilde}\\
    \tilde{F}(X,M)&:= & \dfrac{1}{\sigma^{k \alpha}}F(\sigma ^{k}X, \sigma^{k \alpha} \cdot M+D^{2}P_{k}). \label{def Ftilde}
\end{eqnarray}
Notice that
$$
	\left | D_M \tilde{F}(X,M) - D_M  \tilde{F}(X,N) \right |  \le \omega(\sigma^{k \alpha} \|M-N\|) \le \omega( \|M-N\|),
$$
that is, $\tilde{F}$ fulfills (H2). It readily follows from (\ref{1.3}) that $\tilde{u}$ satisfies
$$
|\tilde{u}| _{L^{\infty }(B_{1})}\leq \delta.
$$
Moreover, $\tilde{u}$ solve
$$
\tilde{F}(X,D^{2}\tilde{u})=\dfrac{1}{\sigma^{k\alpha}}f(\sigma ^{k}X)=:\tilde{f}(X)
$$
in the viscosity sense. From $\tau$-continuity of $f$ and the coefficients of $F$,  together with the smallness condition $\tau(1) \le \delta^{3/2}$, we verify 
$$
	\|\tilde{f}\|_\infty \le \delta^{3/2},
$$
and likewise,
$$
    \sup_{M\in\text{Sym}(n)} \frac{|\tilde{F}(X,M)-\tilde{F}(0,M)|}{\|M\|}  \leq \delta^{3/2}.
$$
Applying Lemma \ref{Key lemma} to $\tilde{u}$ gives a quadratic polynomial $\tilde{P}$
satisfying $\tilde{F}(0,D^{2}\tilde{P})=0$
for which
$$
|\tilde{u}(X)-\tilde{P}(X)|\leq \delta \sigma ^{2+ \alpha},  \quad \text{ for } |X| \leq \sigma.
$$
The $(k+1)$th step of induction is verified if we define
$$
P_{k+1}(X) := P_{k}(X)+\sigma ^{(2+\alpha)k}\tilde{P}(\sigma ^{-k}X).
$$
To conclude the proof of current Lemma, notice  that  \eqref{2.13} implies that 
$$
	\{A_{k}\}\subset \text{Sym}(n), \quad \{b_{k}\}\subset\mathbb{R}^{n},  \quad \mbox{ and }  \quad \{c_{k}\}\subset\mathbb{R}
$$	
are Cauchy sequences. Let us label the limiting quadratic polynomial
$$
    P_{\infty }(X):=\frac{1}{2}X^{t}A_{\infty}X+b_{\infty}X+c_{\infty},
$$
where $A_{k}\rightarrow A_{\infty },b_{k}\rightarrow b_{\infty}$ and $c_{k}\rightarrow c_{\infty }$. It further follows from \eqref{2.13}
\begin{equation}\label{2.14}
    |P_{k}(X)-P_{\infty }(X)|\leq C\delta(\sigma ^{\alpha k}|X|^{2}+\sigma ^{(  1+\alpha)k}|X|+\sigma^{(2+\alpha)k}),
\end{equation}
whenever $|X| \le \sigma^k$. Finally, fixed $X\in B_\sigma,$ take $k \in \mathbb{N}$ such that $\sigma^{k+1}<|X|\leq \sigma ^{k}$ and conclude, by means of  \eqref{1.3} and \eqref{2.14}, that
$$
|u(X)-P_{\infty }(X)| \leq C_{1}\delta\sigma ^{\left( 2+\alpha \right) k}\leq \frac{C_{1}\delta}{\sigma ^{2+\alpha }}|X|^{2+\alpha},
$$
as desired.
\end{proof}

We conclude the proof of Theorem \ref{teo1} by verifying that if $\tau(t) = \tau(1) t^\alpha$,  the smallness condition of Lemma \ref{teo1auxiliar}, namely
$$
 \tau(1) \le \delta^{3/2},
$$
is not restrictive. In fact, if $u\in C^{0}(B_{1})$ is a viscosity solution to
\begin{equation} \label{eq-secstion shauder final}
    F(X,D^{2}u)=f(X)\mbox{ in }B_{1},
\end{equation}
the auxiliary function
$$
v(X):=\frac{u(\mu X)}{\mu ^{2}}
$$
solves
$$
    F_{\mu}(X,D^{2}v)=f_{\mu}(X),
$$
where
$$
    F_{\mu}(X,M):=F(\mu X,M) \quad  \mbox{ and } \quad f_{\mu}(X):=f(\mu X).
$$
Clearly the new operator $F_{\mu}$ satisfies the same assumptions (H1)--(H3) as $F$, with the same universal parameters $\lambda ,\Lambda $ and $\omega$. Note however that
$$
    \max \left \{ |f_{\mu}(X)-f_{\mu}(Y)|, \frac{|F_{\mu}(X,M)-F_{\mu}(Y,M)|}{\|M\|} \right \} \leq \tau(1) \mu^{\alpha}|X-Y|^{\alpha},
$$
for $M \in\text{Sym}(n)$. Thus if $\tau_\mu$ is the modulus of
continuity for $f_\mu$ and $F_\mu$,
$$
    \tau_\mu(1) = \tau(1) \mu^\alpha.
$$
Finally, we take
$$
    \mu :=\min\left\{ 1, \frac{\sqrt[2\alpha]{\delta^3}}{ \sqrt[\alpha]{\tau(1)} }  \right\},
$$
where $\delta$ is the universal number from Lemma \ref{Key lemma}. In conclusion, if $u$ solves \eqref{eq-secstion shauder final}
and satisfies the flatness condition
$$
     \|u\| _{L^{\infty }(B_{1})}\leq \overline{\delta }:=\delta \mu ^{2},
$$
then Lemma \ref{teo1auxiliar} applied to $v$ gives $C^{2,\alpha}$ estimates for $v$, which is transported to $u$ accordantly. \qed

\section{Applications}  \label{sec appl}

Probably an erudite way to comprehend Theorem \ref{teo1} is by saying that if $u$ solves a fully nonlinear elliptic equation with $C^\alpha$ coefficients and source, then if it is close enough to a $C^{2,\alpha}$ function, then indeed $u$ is $C^{2,\alpha}$. This is particularly meaningful in problems involving some {\it a priori} set data. 
 
In this intermediary Section, we  comment on two applications of Theorem \ref{teo1}. The first one concerns an improvement of regularity for classical solutions in H\"older continuous media.

\begin{corollary}[$C^2$ implies $C^{2,\alpha}$]\label{improve reg} Let $u \in C^2(B_1)$ be a classical, poitwise solution to  
$$
	F(X, D^2u) = f(X)
$$ 
where $F(X, \cdot) \in C^1(\text{Sym}(n))$ satisfy (H1)--(H2). Assume further that condition (H3) holds with $\tau(t) = C t^\alpha$ for some $0< \alpha < 1$. Then, $u \in C^{2,\alpha}(B_{1/2})$, and 
$$
    \|u\|_{C^{2,\alpha}(B_{1/2})}  \le C(n,\lambda ,\Lambda, \alpha, \omega, \tau(1), \|u\|_{C^2(B_1)}).
$$
\end{corollary}
\begin{proof}
    We shall proof that $u$ is $C^{2,\alpha}$ at the origin. To this end, define, for an $r>0$ to be chosen soon, $v \colon B_1 \to \mathbb{R}$, by
    $$
        v(X) :=\dfrac{1}{r^2}u(rX) - \left [ \dfrac{1}{r^2}u(0) + \dfrac{1}{r}\nabla u(0)\cdot X + \dfrac{1}{2} X^t D^2u(0) X\right ].
    $$
    We clearly have
    \begin{equation}\label{IR1}
        v(0) = |\nabla v(0)| = 0 \quad \text{and} \quad |D^2v(0)| \le \varsigma(r),
    \end{equation}
    where $\varsigma$ is the modulus of continuity for $D^2u$. Now,  we choose $0< r \ll 1$ so small that
    $$
        \varsigma(r) \le c_n \overline{\delta},
    $$
    where $c_n$ is a dimensional constant and $\overline{\delta}$ is the number appearing in Theorem \ref{teo1}. With such choice, $v$ is under the condition of Theorem \ref{teo1}, for $\tilde{F}(X,M) := F(r X, M  +  D^2u(0))$ and $\tilde{f}(X) = f(rX)$.
\end{proof}

\begin{remark} We remark that  in the proof of Corollary \ref{improve reg},  we can estimate the absolute value of $v$ using integral remainders of the Taylor expansion. Thus, the very same conclusion of that Corollary holds true if we start up only  with VMO condition on $D^2u$. It is also interesting to highlight that Corollary \ref{improve reg} implies that if $u$ is a viscosity solution in $B_1$ of a non-convex, fully nonlinear equation under hypotheses (H1)--(H3). Then if $u$ is $C^2$ at a point $p \in B_1$, then indeed $u$ is $C^{2,\alpha}$ in a neighborhood of $p$. 

\end{remark}
\par 

\medskip 

The second application we explore here regards a mild extension of a recent result due to Armstrong, Silvestre, and Smart \cite{ASS}, on partial regularity for solutions to uniform elliptic PDEs.

\begin{corollary}[Partial regularity] \label{partial reg} Let $u \in C^0(B_1)$ be a viscosity solution to  $F(D^2u) = f(X)$ where $F \in C^1(\text{Sym}(n))$ satisfy $c  \le D_{u_i u_j} F(M) \le c^{-1}$ for some constant $c>0$ and the source function $f$ is Lipschitz continuous. Then, $u \in C^{2,1^{-}}(B_1 \setminus \Sigma)$ for a closed set $\Sigma \subset B_1$, with Hausdorff dimension at most $(n-\epsilon)$ for an $\epsilon > 0$ universal.
\end{corollary}
\begin{proof}
The proof is obtained by similar the reasoning employed in \cite{ASS}. Indeed, the same conclusion of Lemma 5.2 from
\cite{ASS} follows by noticing that  if $f\in C^{0,1}$, then
$$
\mathcal{M}_{\lambda ,\Lambda }^{-}(D^{2}(u_{e}))\leq C \quad  \mbox{ and } \quad \mathcal{M}_{\lambda ,\Lambda }^{+}(D^{2}(u_{e}))\geq -C
$$
where
$$
\mathcal{M}_{\lambda ,\Lambda }^{-}(M):=\inf_{\lambda I_{n}\leq A\leq \Lambda I_{n}}tr(AM), \quad \mathcal{M}_{\lambda ,\Lambda }^{+}(M):=\sup_{\lambda I_{n}\leq A\leq \Lambda I_{n}}tr(AM)
$$
are the Pucci extremal operators. Lemma 7.8 of \cite{CC} can still be employed.  The very same conclusion of Lemma 5.3 from \cite{ASS} also holds true for equations with   Lipschitz sources. Indeed, using the same notation from that Lemma, if  $Y\in B_{\frac{1}{2}}$ is such that there exist $M\in \text{Sym}(n)$, $p\in \mathbb{R}^{n}$ and $Z\in B(Y,r)$  such that
$$
    |u(X)-u(Z)+p.(Z-X)+(Z-X).M(Z-X)|\leq \frac{1}{6}r^{-1}\overline{\delta }|Z-X|^{3}, \quad   X\in B_1,
$$
we define
$$
    v(X)=\frac{1}{16r^{2}}(u(Z+4rX)-u(Z)+4rp.X+16r^{2}X.MX)
$$
and
$$
\widetilde{F}(N)=F(N-M)-F(-M).
$$
Notice that
$$
    \widetilde{F}(X,D^{2}v)=f(Z+4rX)-F(-M)=\widetilde{f}(X)\in C^{0,1}.
$$
Thus, applying Theorem \ref{teo1} to $v$ gives $u$ is $C^{2,1^{-}}$ in $B(Y,r)$. The proof of Theorem \ref{partial reg} follows now exactly as in \cite{ASS}.
\end{proof}

\section{Log-Lipschitz estimates in continuous media} \label{Sect. Log-Lip}

In this Section we proof Theorem \ref{thm LL}. Initially we show that under continuity assumption on the coefficients of $F$ and on the source $f$, after a proper scaling, solutions are under the smallness regime requested by Lemma \ref{Key lemma}, with $\alpha = 0$. For that define
$$
    v(X)=\frac{u(\mu X)}{\mu ^{2}}, \quad F_{\mu}(X,M):=F(\mu X,M) \quad  \mbox{ and } \quad f_{\mu}(X):=f(\mu X),
$$
for a parameter $\mu $ to be determined. Equation
$$
    F_{\mu}(X,D^{2}v)=f_{\mu}(X),
$$
is satisfied in the viscosity sense. Now we choose $\mu$ so small  that
$$
     \tau (\mu)\leq \delta^{3/2},
$$
where $\tau$ is the modulus of continuity of the media and $\delta > 0$ is the number sponsored by Lemma \ref{Key lemma} with $\alpha =0$. 
In the sequel, define
$$
    \tau _{{\mu }}(t):=\tau(\mu t)
$$
and note that
$$
\max \left \{ |f_{\mu}(X |, \frac{|F_{\mu}(X,M)-F_{\mu}(0,M)|}{\|M\|} \right \} \leq \tau_{{\mu}}(|X-Y|).
$$
Thus,
$$
\sup_{M\in  \text{Sym}(n)}  \frac{|F_{\mu}(X,M)-F_{\mu }(0,M)|}{\|M\|} \leq \delta^{3/2} \quad \mbox{ and }\quad  \|f_{\mu}\|_{L^{\infty }(B_{1})}\leq  \delta^{3/2}.
$$
Now if we take
$$
\|u\|_{L^{\infty}(B_{1})}\leq \overline{\delta}:= \delta \mu^{2}
$$
then
$$
\|v\|_{L^{\infty}(B_{1})}\leq \delta,
$$
Estimates proven for $v$ gives the desired ones for $u$. 
\par
\medskip

The conclusion of the above reasoning is that we can start off the proof of Theorem \ref{thm LL} out from Lemma \ref{Key lemma}. That is, the proof of the current Theorem begins with the existence  of a quadratic polynomial $P_1$ satisfying
$F(0,D^{2}P_1)=0$  and a number $\sigma>0$ for which the following estimate
\begin{equation}\label{LL-inductionk=1}
    \sup_{B_{\sigma }}|u-P_1| \leq\sigma ^{2}\delta,
\end{equation}
holds, provided $\delta$ is small enough, depending only on universal parameters. As in Lemma \ref{teo1auxiliar}, we shall prove by induction process the existence of a sequence of polynomials 
$$
	P_{k}(X)=\frac{1}{2}X^{t}A_{k}X+b_{k}X+c_{k}
$$ 
satisfying $F(0,D^{2}P_{k})=0$ such that 
\begin{equation}
|u(X)-P_{k}(X)|\leq \delta\sigma ^{2k} \quad \mbox{ for }|X|\leq \sigma ^{k}.\label{l-l ref 1}
\end{equation}
Moreover, we have the following estimates on the coefficients
\begin{equation}
\left\{
\begin{array}{rrl}
|A_{k}-A_{k-1}| & \leq & C\delta \\  \label{l-l ref 2}
|b_{k}-b_{k-1}| & \leq & C\delta \sigma^{(k-1)} \\
|c_{k}-c_{k-1}| & \leq & C\delta \sigma^{2(k-1)}.
\end{array}\right.
\end{equation}
The case $k=1$ is precisely the conclusion enclosed in \eqref{LL-inductionk=1}. Assume we have verified the $k$th step of induction. Define the scaled function and the scaled operator
$$
    \tilde{u}(X):=\dfrac{1}{\sigma ^{2k}}(u(\sigma ^{k}X)-P_{k}(\sigma^{k}X)) \quad  \mbox{ and } \quad \tilde{F}(X,M):=F(\sigma ^{k}X,M+D^{2}P_{k}).
$$
Easily one verifies that $\tilde{u}$ is a viscosity solution to
$$
\tilde{F}(X,D^{2}\tilde{u})=f(\sigma ^{k}X):=\tilde{f}(X).
$$
From the induction hypothesis, (\ref{l-l ref 1}),  $\tilde{u}$ is flat, i.e.,  $|\tilde{u}|_{L^{\infty }(B_{1})}\leq \delta$. Also, clearly
$$
 \sup_{M \in \text{Sym}(n)} \frac{|\tilde{F}(X,M)-\tilde{F}(0,M)|}{\|M\|} \leq \delta^{3/2} \quad \mbox{ and } \quad \|\tilde{f}\|_{L^{\infty}(B_{1})}\leq  \delta^{3/2}.
$$
That is, $\tilde{u}$ is entitled to the conclusion \eqref{LL-inductionk=1}, thus there exists a quadratic polynomial $\tilde{P}$ with $\tilde{F}(0,D^{2}\tilde{P})=0$
and
$$
    |\tilde{u}(X)-\tilde{P}(X)|\leq \delta \sigma ^{2k} \quad \mbox{ for }|X|\leq\sigma.
$$
The $(k+1)$th step of induction follows by defining
$$
    P_{k+1}(X):= P_{k}(X)+\sigma^{2k}\tilde{P}(\sigma ^{-k}X).
$$
In view of the coefficient oscillation control \eqref{l-l ref 2}, we conclude $b_k$ converges in $\mathbb{R}^n$ to a vector $b_\infty$ and $c_k$ converges in $\mathbb{R}$ to a real number $c_\infty$. Also
\begin{eqnarray}
    |c_{k} -  c_{\infty}| &\le&  C\delta \sigma^{2k}, \label{LL-est001} \\
    |b_{k} - b_{\infty }|  &\le&  C\delta \sigma^{k}.   \label{LL-est002}
\end{eqnarray}
The sequence of matrices $A_k$ may diverge, however, we can at least estimate
\begin{equation}\label{BMO on D2}
    \|A_{k}\|_{\text{Sym(n)}} \leq kC\delta.
\end{equation}
In the sequel, we define the tangential affine function
$$
\ell_{\infty }(X): =c_{\infty }+b_{\infty }\cdot X
$$
and estimate, in view of \eqref{LL-est001}, \eqref{LL-est002} and \eqref{BMO on D2}, for $|X| \le \sigma^k$,
\begin{equation} \label{LL-final0}
\begin{aligned}
|u(X)-\ell_{\infty }(X)| &\leq|u(X)-P_{k}(X)|+ |c_{k}-c_{\infty }| +|(b_{k}-b_{\infty })||X|+|A_{k}||X|^{2} \\
&\leq\delta \sigma ^{2k}+2C\delta \sigma ^{2k}+kC\delta \sigma ^{2k} \\
& \le  C\delta (k \sigma^{2k}).
\end{aligned}
\end{equation}
Finally, fixed $X\in B_\sigma$, take $k \in \mathbb{N}$ such that $\sigma^{k+1}<|X|\leq \sigma ^{k}$. From  \eqref{LL-final0}, we find
$$
|u(X)-\ell_{\infty }(X)| \leq - ({C_{1}\delta} ) \cdot |X|^{2}\log{|X|},
$$
as desired. The proof of Theorem \ref{thm LL} is concluded. \qed


\begin{thebibliography}{12}
\bibitem{ASS} Armstrong S., Silvestre, L. and Smart, C. \textit{Partial regularity of solutions of fully nonlinear uniformly elliptic
equations}  Comm. Pure Appl. Math. {\bf 65} (2012), no. 8, 1169--1184.


\bibitem{CC03} Cabr\'e, Xavier; Caffarelli, Luis A.  \textit{Interior $C^{2,\alpha}$ regularity theory for a class of nonconvex fully nonlinear elliptic equations.}    J. Math. Pures Appl.    \textbf{82} (9) (2003), 573--612

\bibitem{C1} Caffarelli, Luis A. \textit{Interior a priori estimates for solutions of fully nonlinear equations.} Ann.
of Math. (2) \textbf{130} (1989), no. 1, 189--213.

\bibitem{CC} Caffarelli, Luis A.; Cabr\'e, Xavier
\textit{Fully nonlinear elliptic equations.} American Mathematical
Society Colloquium Publications, 43. American Mathematical Society,
Providence, RI, 1995.


\bibitem{CY} Caffarelli, Luis A.; Yuan, Yu {\it A priori estimates for solutions of fully nonlinear equations with convex level set.}
Indiana Univ. Math. J. {\bf 49} (2000), no. 2, 681--695.

\bibitem{NV1} N. Nadirashvili and S. Vladut,
\textit{Nonclassical solutions of fully nonlinear elliptic equations}. Geom. Funct. Anal. {\bf 17} (2007), no. 4, 1283--1296.

\bibitem{NV2} N. Nadirashvili and S. Vladut,
{\it Singular viscosity solutions to fully nonlinear elliptic
equations.} J. Math. Pures Appl. (9) {\bf 89} (2008), no. 2,
107--113.


\bibitem{S} Savin, O. {\it Small perturbation solutions for elliptic equations.}
Comm. Partial Differential Equations {\bf 32} (2007), no. 4-6, 557--578.


\bibitem{T1} Teixeira, Eduardo V. \textit{Universal moduli of continuity for
solutions to fully nonlinear elliptic equations.} To appear in  Arch. Ration. Mech. Anal.



\end{thebibliography}
\end{document}